\documentclass[12pt]{article}
\usepackage[utf8]{inputenc}
\usepackage{amsmath}
\usepackage{bbm}
\usepackage{mathtools}
\usepackage{amsfonts}
\usepackage{dsfont}
\usepackage{amssymb}
\usepackage[margin=1in,headheight=1in, asymmetric]{geometry}
\usepackage{amsthm}
\usepackage{color}

\newtheorem{theorem}{Theorem}

\newtheorem{lemma}[theorem]{Lemma}
\newtheorem{definition}[theorem]{Definition}

\usepackage{comment}
\usepackage{setspace}

\newcommand{\tr}{\mathrm{T}}

\title{The Fan--Taussky--Todd inequalities and the Lumer--Phillips theorem}
\author{Benedict Bauer\thanks{Financial support from the Austrian Science Fund (FWF) under grants P~30750
and Y~1235 is gratefully acknowledged.}\\
University of Vienna \\
\tt{benedict.bauer@univie.ac.at}\\
\\
 Stefan Gerhold \\
TU Wien \\
\tt{sgerhold@fam.tuwien.ac.at}
}

\date{\today}

\numberwithin{equation}{section}
\numberwithin{theorem}{section}

\begin{document}

\maketitle

\begin{abstract}

We argue that a classical inequality due to Fan, Taussky and Todd (1955) is equivalent
to the dissipativity of a Jordan block. As the latter can be characterised via
the zeros of Chebyshev polynomials, we obtain a short new proof of the inequality.
Three other inequalities of Fan--Taussky--Todd are reproven similarly. By the Lumer--Phillips theorem,
the semigroup defined by the Jordan block is contractive. This yields new extensions
of the classical Fan--Taussky--Todd inequalities.
 As an application, we give an estimate for the partial sums of a Bessel function.
\end{abstract}

MSC 2020: 26D15, 33C45, 15A16, 33C10
\smallskip

Keywords: Fan-Taussky-Todd inequalities, Jordan block, tridiagonal matrix, Chebyshev polynomials, Lumer-Phillips theorem, modified Bessel function of the first kind

\section{Introduction}

In 1955, Fan, Taussky and Todd proved the following two theorems (use the identity $\cos 2\theta=1-2\sin^2\theta$ to get the precise form stated
in~\cite{FaTaTo55}):
\begin{theorem}\cite[Theorem~9]{FaTaTo55}\label{thm:ftt1}
  For real numbers $a_1,\dots,a_n$, with $a_0:=a_{n+1}:=0$, we have
  \begin{equation}\label{eq:ftt1}
    \sum_{k=1}^{n+1}(a_k-a_{k-1})^2 \geq 2\Big(1-\cos \frac{\pi}{n+1}\Big) \sum_{k=1}^n a_k^2.
  \end{equation}
\end{theorem}
\begin{theorem}\cite[Theorem~8]{FaTaTo55}\label{thm:ftt2}
  For real numbers $a_1,\dots,a_n$, with $a_0:=0$, we have
  \begin{equation}\label{eq:ftt2}
    \sum_{k=1}^{n}(a_k-a_{k-1})^2 \geq 2\Big(1-\cos \frac{\pi}{2n+1}\Big) \sum_{k=1}^n a_k^2.
  \end{equation}
\end{theorem}
The following converse inequalitites are special cases
of results by Milovanovi{\'c} and Milovanovi{\'c}~\cite{MiMi82}.
\begin{theorem}\label{thm:ftt conv1}
  For real numbers $a_1,\dots,a_n$, with $a_0:=a_{n+1}:=0$, we have
  \begin{equation}\label{eq:ftt conv1}
    \sum_{k=1}^{n+1}(a_k-a_{k-1})^2 \leq 2\Big(1+\cos \frac{\pi}{n+1}\Big) \sum_{k=1}^n a_k^2.
  \end{equation}
\end{theorem}
\begin{theorem}\label{thm:ftt conv2}
  For real numbers $a_1,\dots,a_n$, with $a_0:=0$, we have
  \begin{equation}\label{eq:ftt conv2}
    \sum_{k=1}^{n}(a_k-a_{k-1})^2 \leq 2\Big(1-\cos \frac{2\pi}{2n+1}\Big) \sum_{k=1}^n a_k^2.
  \end{equation}
\end{theorem}
Alzer~\cite{Al91} gave short proofs of Theorems~\ref{thm:ftt conv1}
and~\ref{thm:ftt conv2}.
In this note we give  new short proofs of Theorems~\ref{thm:ftt1}
and~\ref{thm:ftt conv1}, and analogous proofs of
 Theorems~\ref{thm:ftt2} and~\ref{thm:ftt conv2}. There are several proofs
 in the literature~\cite{Lu94,Re83,ZhXi05}.
Our approach is symmetric in the sense that our proofs of Theorems~\ref{thm:ftt1}
and~\ref{thm:ftt conv1} are trivial modifications of each other, and the same 
holds for Theorems~\ref{thm:ftt2} and~\ref{thm:ftt conv2}.
The proofs are based on the dissipativity of Jordan blocks. 
Besides reproving inequalities which are already known (Section~\ref{se:proofs}),
we note in Section~\ref{se:gen} that dissipativity characterizes contractiveness of the matrix
semigroup generated by the Jordan block, by the Lumer--Phillips theorem.
This  leads to  new generalizations
of Theorems~\ref{thm:ftt1} and~\ref{thm:ftt2}.
These generalizations are then applied to estimate
partial sums of the modified Bessel function of the first kind
of order zero.
In Section~\ref{se:strict}, we prove strict versions of the new inequalities.

\section{Proofs}\label{se:proofs}

Define the Jordan block
\[
  J_n(x):= \begin{pmatrix}
  x & 1 & & 0 \\
   & x & 1  & \\
  &  & \ddots  & 1\\
  0 & &  & x
  \end{pmatrix} \in \mathbb{R}^{n\times n}, \quad x\in\mathbb R.
\]
We use the following terminology from operator theory, and, as usual,
reserve the notion of semidefinitess for symmetric matrices.
\begin{definition} (see~\cite[p.~52]{Da80})
  A matrix $A\in\mathbb{R}^{n\times n}$ is \emph{dissipative},
  if $\langle A a,a\rangle\leq 0$ for all $a\in\mathbb{R}^n$.
\end{definition}
Since
\begin{align*}
 \langle  J_n(\alpha)a, a\rangle = \alpha \sum_{k=1}^n a_k^2+ \sum_{k=2}^n a_k a_{k-1},
  \quad a=(a_1,\dots,a_n)^\tr \in\mathbb{R}^n,
\end{align*}
it is clear that~\eqref{eq:ftt1} and~\eqref{eq:ftt conv1} are immediate
consequences of the following lemma.
\begin{lemma}\label{le:main}
  Let $\alpha\in\mathbb R$.
  
  (i) $J_n(\alpha)$ is dissipative if and only 
  if $\alpha \leq -\cos(\pi/(n+1))$.
  
  (ii) $-J_n(\alpha)$ is dissipative  if and only 
  if $\alpha \geq \cos(\pi/(n+1))$.
\end{lemma}
\begin{proof}
  Since $2a^\tr J_n(\alpha) a = a^\tr J_n(\alpha)^\tr a + a^\tr J_n(\alpha) a $, $a\in\mathbb{R}^n$,
  the first condition in~(i) is equivalent to $B_n(\alpha):= J_n(\alpha)^\tr+ J_n(\alpha)$ being 
  negative semidefinite, i.e.\ all its eigenvalues being non-positive.
  It is well-known~\cite[pp.~25-26]{Ch78} that $\det B_n(x) = U_n(x)$, $x\in\mathbb R$,
  where $U_n(x)$ is the $n$th Chebyshev polynomial of the second kind. Thus, any eigenvalue~$\mu$
  of~$B_n(\alpha)$ must satisfy
  \[
    0=\det(B_n(\alpha)-\mu I)= \det B_n(\alpha-\tfrac12 \mu) = U_n(\alpha -\tfrac12 \mu).
  \]
  Now note that all such $\mu$ are $\leq 0$ if and only if
  \[
    \alpha \leq \min_{1\leq k\leq n}\cos \frac{k\pi}{n+1}
    =\cos \frac{n\pi}{n+1} = - \cos \frac{\pi}{n+1},
  \]
  where $\cos(k\pi/(n+1))$ are the zeros of $U_n$. The proof of~(ii) is analogous;
  now~$\alpha$ must satisfy
  \[
    \alpha \geq \max_{1\leq k\leq n}\cos \frac{k\pi}{n+1} = \cos \frac{\pi}{n+1}.
    \qedhere
  \] 
\end{proof}
To prove Theorems~\ref{thm:ftt2} and~\ref{thm:ftt conv2}, define
\[
  \tilde{J}_n(x):= \begin{pmatrix}
  x & 1 & & &0 \\
   & x & 1  & &\\
  &  & \ddots  & 1 &\\
  & &  & x & 1\\
  0& & &  & x-\tfrac12
  \end{pmatrix} \in \mathbb{R}^{n\times n}, \quad x\in\mathbb R.
\]
%
\begin{lemma}\label{le:det} 
  For $x\in\mathbb R$ and $n\in\mathbb N$, we have the  determinant evaluation
  $\det (\tilde{J}_n(x)^\tr+\tilde{J}_n(x))=U_n(x) - U_{n-1}(x)$.
  As above, $U_n(x)$ denotes the $n$th Chebyshev polynomial of the second kind.
\end{lemma}
\begin{proof}
 By a classical result on tridiagonal matrices~\cite[Chapter~XIII]{Mu60}, this determinant is~$b_n=b_n(x)$, where the sequence
  $(b_k)_{-1\leq k\leq n}$ is defined by $b_{-1}=0$, $b_0=1$, and
  \begin{align*}
     b_k &= 2x b_{k-1} - b_{k-2},\quad 1 \leq k < n,\\
     b_n &= (2x-1)b_{n-1}-b_{n-2}.
  \end{align*}
  For $k<n$, this is the recurrence for the Chebyshev polynomials of the second kind,
  and so $b_k=U_k$ for $k<n$. Finally,
  \[
    b_n = (2x-1)U_{n-1}-U_{n-2}=U_n - U_{n-1}. \qedhere
  \]
\end{proof}
\begin{lemma}\label{le:zeros}
  For $n\in\mathbb N$, the zeros of $U_n(x) - U_{n-1}(x)$ are $(-1)^{k+1}\cos(k\pi/(2n+1))$,
  $1\leq k \leq n$.
\end{lemma}
\begin{proof}
  Let~$k\in\{1,\dots,n\}$ be even. We have
  \[
    -\cos \frac{k\pi}{2n+1} = \cos\Big(\pi - \frac{k\pi}{2n+1}\Big).
  \]
  The assertion now follows from the representation
  \[
    U_n(\cos \theta)= \frac{\sin((n+1)\theta)}{\sin \theta},\quad \theta \in \mathbb R,
  \]
  because
  \begin{align*}
    n \Big(\pi - \frac{k\pi}{2n+1}\Big) &=
      \Big(n -\frac{k}{2}+\frac12\Big)\pi - \frac{(2n-k+1)\pi}{4n+2}, \\
      (n+1) \Big(\pi - \frac{k\pi}{2n+1}\Big) &=
      \Big(n -\frac{k}{2}+\frac12\Big)\pi + \frac{(2n-k+1)\pi}{4n+2}.
  \end{align*}
  The proof for odd~$k$ is analogous.
\end{proof}

Since
\[
  \min_{1\leq k\leq n}(-1)^{k+1} \cos \frac{k\pi}{2n+1}= -\cos \frac{2\pi}{2n+1}
\]
and
\[
  \max_{1\leq k\leq n}(-1)^{k+1} \cos \frac{k\pi}{2n+1}= \cos \frac{\pi}{2n+1},
\]
Lemmas~\ref{le:det} and~\ref{le:zeros} show that the following result can be proven analogously
to Lemma~\ref{le:main}.
\begin{lemma}\label{le:main2}
  Let $\alpha\in\mathbb R$.
  
  (i)  $\tilde{J}_n(\alpha)$ is dissipative
   if and only 
  if $\alpha \leq -\cos(2\pi/(2n+1)$.
  
  (ii)  $-\tilde{J}_n(\alpha)$ is dissipative
    if and only 
  if $\alpha \geq \cos(\pi/(2n+1)$.
\end{lemma}
For $a\in\mathbb{R}^n$, we have
\[
  a^\tr \tilde{J}_n(\alpha) a = \alpha \sum_{k=1}^n a_k^2 -\tfrac12 a_n^2+ \sum_{k=2}^n a_k a_{k-1}.
\]
Now~\eqref{eq:ftt2} and~\eqref{eq:ftt conv2} easily follow, by applying Lemma~\ref{le:main2} (i), (ii)
with the respective optimal values of~$\alpha$.

\section{A generalization}\label{se:gen}

In what follows, we consider the Euclidean norm on~$\mathbb{R}^n$, and
write
\[
  \| A \|_{\mathrm{op}} = \sup_{0\neq a\in \mathbb{R}^n} \frac{\| Aa \|}{\| a\|}
\]
for the operator norm of an $n\times n$ matrix~$A$. The contractivity
of matrix semigroups w.r.t.\ this norm is characterized by the Lumer--Phillips
theorem~\cite[p.~52]{Da80}.
In our setting, i.e.\ for semigroups on~$\mathbb{R}^n$, it is the following statement.
\begin{theorem}\label{thm:lp}
  For any real $n\times n$ matrix~$Q$, the following are equivalent:
  
  (i) $\| \exp(Qx) \|_{\mathrm {op}}\leq 1$ for all $x \geq 0$,
  
  (ii) $Q$ is dissipative.
\end{theorem}
This yields a generalization of Theorem~\ref{thm:ftt1}, involving
an extra parameter $x\geq0$.
\begin{theorem}\label{thm:gftt}
   For real numbers $a_1,\dots,a_n$ and $x\geq0$, we have
  \begin{equation}\label{eq:gen ftt}
     \sum_{j=0}^{n-1}\bigg( \sum_{k=0}^j \frac{x^k}{k!}a_{n-j+k}\bigg)^2
  \leq \exp\Big(2x \cos\Big(\frac{\pi}{n+1}\Big) \Big) \sum_{j=1}^n a_j^2.
  \end{equation}
\end{theorem}
\begin{proof}
By Lemma~\ref{le:main}, $J_n(\alpha)$ satisfies condition~(ii)
of Theorem~\ref{thm:lp} for
 $\alpha=-\cos(\pi/(n+1)$, and so
\begin{equation}\label{eq:2 est}
  \| \exp(J_n(\alpha)x)a \|^2 \leq \|a\|^2, \quad x\geq 0,\ a\in\mathbb{R}^n.
\end{equation}
By calculating the matrix exponential of a nilpotent matrix (see~\cite[p.~9]{EnNa00}),
\[
  \exp(x(J_n(1)-I))=
  \exp
   \begin{pmatrix}
  0 & x & & 0 \\
   & 0 & x  & \\
  &  & \ddots  & x\\
  0 & &  & 0
  \end{pmatrix}
  =
  \begin{pmatrix}
  1 & x & \frac{x^2}{2} & \dots & \frac{x^{n-1}}{(n-1)!} \\
    & 1 & x & \dots &  \frac{x^{n-2}}{(n-2)!}\\
    & & \ddots \\
    & & & & 1
  \end{pmatrix},
  \quad x\in\mathbb R,
\]
we find
\[
  \exp(J_n(\alpha)x) = \exp(x(J_n(1)-I)+\alpha x I)
  =e^{\alpha x}
   \begin{pmatrix}
  1 & x & \frac{x^2}{2} & \dots & \frac{x^{n-1}}{(n-1)!} \\
    & 1 & x & \dots &  \frac{x^{n-2}}{(n-2)!}\\
    & & \ddots \\
    & & & & 1
  \end{pmatrix}.
\]
Using this in~\eqref{eq:2 est} yields~\eqref{eq:gen ftt}, because
\begin{align*}
  \| \exp(J_n(\alpha)x)a \|^2 &= e^{2\alpha x}\sum_{m=1}^n
    \bigg( \sum_{k=m}^n a_k \frac{x^{k-m}}{(k-m)!}\bigg)^2\\
    &= e^{2\alpha x}\sum_{j=0}^{n-1}
    \bigg( \sum_{k=n-j}^n a_k \frac{x^{k+j-n}}{(k+j-n)!}\bigg)^2
    =e^{2\alpha x}\sum_{j=0}^{n-1}
    \bigg( \sum_{k=0}^j a_{n-j+k} \frac{x^{k}}{k!}\bigg)^2. \qedhere
\end{align*}
\end{proof}
As~\eqref{eq:gen ftt} is obviously sharp for $x\downarrow 0$, the inequality must hold after taking the derivative
w.r.t.~$x$ at zero on both sides. An easy calculation shows that this yields~\eqref{eq:ftt1}, and so
Theorem~\ref{thm:gftt} can be viewed as a generalization of Theorem~\ref{thm:ftt1}. Another special
case might be worth mentioning: Putting $a=(0,\dots,0,1)$ in~\eqref{eq:gen ftt} yields
\begin{equation}\label{eq:bessel1}
 \sum_{j=0}^{n-1} \frac{x^{2j}}{j!^2} \leq
   \exp\Big(2x \cos\Big(\frac{\pi}{n+1}\Big) \Big),
  \quad n\in\mathbb N,\ x\geq 0.
\end{equation}
This inequality for the partial sums of the modified Bessel function of the first kind
$I_0(2x)= \sum_{j=0}^\infty \frac{x^{2j}}{j!^2}$ seems to be new.
Analogously to Theorem~\ref{thm:gftt}, we can generalize Theorem~\ref{thm:ftt2} as follows.
\begin{theorem}\label{thm:gftt2}
   For real numbers $a_1,\dots,a_n$ and $x\geq0$, we have
  \begin{equation*}
     \sum_{j=1}^{n-1}\bigg( \sum_{k=0}^j \frac{x^k}{k!}a_{n-j+k}\bigg)^2
     +e^{-x}a_n^2
  \leq \exp\Big(2x \cos\Big(\frac{2\pi}{2n+1}\Big) \Big) \sum_{j=1}^n a_j^2.
  \end{equation*}
\end{theorem}
\begin{proof}
  The proof is very similar to that of Theorem~\ref{thm:gftt}.
  Use Lemma~\ref{le:main2}, with $\alpha = -\cos(2\pi/(2n+1)$,
  and calculate $\| \exp( \tilde{J}_n(\alpha)x)a \|^2$
  in order to apply Theorem~\ref{thm:lp} .
\end{proof}
Taking the derivative of~\eqref{eq:gen ftt} at zero yields~\eqref{eq:ftt2} (cf.\ the remark after
Theorem~\ref{thm:gftt}).
For  $a=(0,\dots,0,1)$, as above, \eqref{eq:gen ftt} yields
\begin{equation}\label{eq:bessel2}
 \sum_{j=0}^{n-1} \frac{x^{2j}}{j!^2} \leq 1-e^{-x}+
   \exp\Big(2x \cos\Big(\frac{2\pi}{2n+1}\Big) \Big),
  \quad n\in\mathbb N,\ x\geq 0.
\end{equation}
For fixed~$n$, the estimate~\eqref{eq:bessel2} is sharper
than~\eqref{eq:bessel1} for large~$x$. We conjecture that there is
a threshold $x_0=x_0(n)>0$ such that~\eqref{eq:bessel1} gives a better bound
for $0<x<x_0$.

\section{Strict inequalities}\label{se:strict}

We were not able the find the following strict version of the Lumer--Phillips
theorem in the literature:
\begin{theorem}\label{thm:lp strict}
    For any real $n\times n$ matrix $Q$, the following are equivalent:
    
    (i) $\|\exp(Qx)\|_{\operatorname{op}}< 1$ for all $x>0$,
    
    (ii) $\langle Qa,a \rangle < 0$ for all $a \in \mathbb{R}^n\setminus \{0\}$.
\end{theorem}

Analogously to Theorem~\ref{thm:gftt}, Theorem~\ref{thm:lp strict} implies the following
statement, and an analogous strict variant of Theorem~\ref{thm:gftt2}.
\begin{theorem}
   For real numbers $a_1,\dots,a_n$,  not all zero, $x>0$ and $\alpha>\cos(\pi/(n+1))$, we have
  \begin{equation*}
     \sum_{j=0}^{n-1}\bigg( \sum_{k=0}^j \frac{x^k}{k!}a_{n-j+k}\bigg)^2
  < e^{2x\alpha} \sum_{j=1}^n a_j^2.
  \end{equation*}
\end{theorem}
It remains to prove Theorem~\ref{thm:lp strict}.
\begin{lemma}
\label{constNorm}
Let $Q$ be a real $n\times n$ matrix satisfying
     $\|\exp(Qx)\|_{\operatorname{op}}\leq 1$ for all $x\geq 0$.
    Assume there exists $x_0>0$ such that $\|\exp(Qx_0)\|_{\operatorname{op}}=1$. Then $\|\exp(Qx)\|_{\operatorname{op}}=1$ for all $x\geq 0$.
\end{lemma}
\begin{proof}
    Choose $a \in \mathbb{R}^n$ with unit length such that $\|\exp(Qx_0)a\|=1$. Clearly we must have $\|\exp(Qx)a\|=1$ for all $x \in [0,x_0]$, since otherwise, by submultiplicativity,
      \begin{align*}
      \|\exp(Qx_0)a\|&=\|\exp(Q(x_0-x))\exp(Qx)a\|\\
      & \le  \underbrace{\|\exp(Q(x_0-x))\|}_{\le 1} \ \underbrace{\|\exp(Qx)a\|}_{<1} <1.
      \end{align*}
       Now $\langle \exp(Qx)a,\exp(Qx)a\rangle-1$ is analytic in $x$ and vanishes on $[0,x_0]$. Hence it must vanish everywhere, and we have $\|\exp(Qx)a\|=1$ for all $x$.
\end{proof}
\begin{lemma}
    \label{leNorm}
    Let $Q$ be a real $n\times n$ matrix satisfying
     $\|\exp(Qx)\|_{\operatorname{op}}\leq 1$ for all $x\geq 0$.
    The set of vectors $a\in\mathbb{R}^n$ satisfying $\|\exp(Qx)a\|=\|a\|$ for all $x \ge 0$ is a subspace.
     Denote this subspace by $W_Q$. Then~$W_Q$ and~$W^\perp_Q$ are invariant under $\exp(Qx)$.
\end{lemma}
\begin{proof}
   For $a\in\mathbb{R}^n$ and $x\geq 0,$ we have
   \begin{equation*}
       \|\exp(Qx)a\|=\|a\| \quad \Longleftrightarrow \quad
       \big\langle (\operatorname{id}-\exp(Q^\tr x)\exp(Qx))a,a\big\rangle =0.
   \end{equation*}
     Our assumption implies that $\operatorname{id}-\exp(Q^\tr x)\exp(Qx)$ is 
   positive semidefinite, and thus the condition $ \|\exp(Qx)a\|=\|a\|$ is equivalent to
   $a\in\operatorname{ker}(\operatorname{id}-\exp(Q^\tr x)\exp(Qx))$.
   By Lemma~\ref{constNorm}, this kernel is the same for every $x>0$.
      Let $u \in W^\perp_Q$. Then, for any $w \in W_Q$ we have
    \[
    \langle \exp(Qx)u,w\rangle =\langle u,\exp(Qx)^\tr w\rangle= \langle u,\exp(-Qx)w\rangle=0,
    \]
    since $\exp(-Qx)w \in W_Q$. Hence $W_Q^\perp$ is invariant under $\exp(Qx)$.
    Invariance of $W_Q$ is clear.
\end{proof}
\begin{proof}[Proof of Theorem~\ref{thm:lp strict}]
    If~(ii) holds, then  $\|\exp(Qx)\|_{\operatorname{op}} = 1$, $x\geq0$, by Theorem~\ref{thm:lp}.
     Assume for the sake of contradiction that $\|\exp(Qx)\|_{\operatorname{op}} = 1$ for some $x>0$.
    By Lemmas~\ref{constNorm} and~\ref{leNorm}, there exists $0\neq a\in\mathbb{R}^n$ such that $\langle \exp(Qx)a,\exp(Qx)a\rangle$ is constant in $x$.
 Hence the derivative at $x=0$ has to be zero, and hence $\langle Qa,a \rangle = 0$, contradicting our assumption.
    Conversely, assume there exists $a \neq 0$ with $\langle Qa,a \rangle = 0$. Then we have $a\in\operatorname{ker}(Q+Q^\tr) $, since $Q+Q^\tr$ is negative semidefinite. Hence $\exp(Q^*x)a = \exp(-Qx)a$, and we have
    \begin{align*}
    \langle \exp(Qx)a,\exp(Qx)a \rangle &=\langle a,\exp(Q^\tr x)\exp(Qx)a \rangle\\
     &=\langle a,\exp(-Qx)\exp(Qx)a \rangle = \langle a,a \rangle.
    \end{align*}
    Therefore, $\|\exp(Qx)\|_{\operatorname{op}} = 1$ for all $x\geq 0$.
\end{proof}

\bibliographystyle{siam}
\bibliography{../literature}

\end{document}